\documentclass[a4paper, 12pt, titlepage, fleqn]{article}
\usepackage{times}
\usepackage{amsthm}
\newtheorem{thm}{Theorem}
\newtheorem{cor}{Corollary}
\newtheorem{conj}{Conjecture}
\newtheorem{lemma}[thm]{Lemma}
\newtheorem{prop}{Proposition}

\usepackage{amssymb,amsmath}

\DeclareMathOperator{\F}{\mathbb{F}}

\DeclareMathOperator{\Tr}{Tr}

\begin{document}
		\baselineskip=16.3pt
		\parskip=14pt
		\begin{center}
			\section*{Divisibility of L-Polynomials for a Family of Artin-Schreier Curves}
			{\large 
			Gary McGuire\footnote{email gary.mcguire@ucd.ie, Research supported by Science Foundation Ireland Grant 13/IA/1914} and
			Emrah Sercan Y{\i}lmaz \footnote {Research supported by Science Foundation Ireland Grant 13/IA/1914} 
			\\
			School of Mathematics and Statistics\\
			University College Dublin\\
			Ireland}

		\end{center}
		
		\subsection*{Abstract}
In this paper we consider the curves $C_k^{(p,a)} : y^p-y=x^{p^k+1}+ax$ defined
over $\mathbb F_p$ and give a positive answer to a conjecture  about a divisibility condition on $L$-polynomials of the curves $C_k^{(p,a)}$. 
Our proof involves finding an exact formula for the number of $\F_{p^n}$-rational points
on $C_k^{(p,a)}$ for all $n$, and uses a result we proved
elsewhere about the number of rational points on supersingular curves.

\section{Introduction}
Let $p$ be a prime and let $q=p^r$ where $r$ is a positive integer. 
Let $\mathbb{F}_q$ be the finite field with $q$ elements.
Let $X$ be a projective smooth absolutely irreducible curve of genus $g$ 
defined over $\mathbb{F}_q$.
The $L$-polynomial of the curve $X$ over $\mathbb F_{q}$ is defined by
$$L_{X/\mathbb{F}_q}(T)=L_X(T)=\exp\left( \sum_{n=1}^\infty ( \#X(\mathbb F_{q^n}) - q^n - 1 )\frac{T^n}{n}  \right).$$
where $\#X(\mathbb F_{q^n})$ denotes the number of $\mathbb F_{q^n}$-rational points of $X$. 
It is well known that $L_X(T)$ is a polynomial of degree $2g$ with integer coefficients, so we write it as 
\begin{equation} \label{L-poly}
L_X(T)= \sum_{i=0}^{2g} c_i T^i, \ c_i \in \mathbb Z.
\end{equation}
It is also well known that $c_0=1$ and $c_{2g}=q^g$.

We wish to consider the question of divisibility of L-polynomials.
In previous papers \cite{chapman:AM}, \cite{chapman:AMR},
we have studied conditions on the curves under which the L-polynomial of
one curve divides the L-polynomial of another curve.
A theorem of Tate gives an answer in terms of Jacobians.
We refer the reader to these papers for a longer discussion of this topic.

Artin-Schreier curves are degree $p$ coverings of the projective line,
and are cyclic extensions of degree $p$ of the rational function field.
It can be shown that all Artin-Schreier curves have an equation of the
form $y^p-y=f(x)$.
Let $k$ be a positive integer. In this paper we will study 
the family of Artin-Schreier curves $$C_k^{(p,a)} : y^p-y=x^{p^k+1}+ax$$ 
where $a\in \F_p$,
which are defined over $\mathbb F_p$
and have genus $p^k(p-1)/2$.
We will prove the following conjecture, which is stated in \cite{conj4}.
\begin{conj}\label{conj}
	Let $k$ and $m$ be  positive integers. Then the $L$-polynomial of $C_{km}^{(p,a)}$ is divisible by the $L$-polynomial 
	of $C_k^{(p,a)}$.
\end{conj}

The L-polynomials in the conjecture are over $\F_p$.
This conjecture was proved for $p=2$ in \cite{conj4}, so we will assume that $p$ is odd for this paper.
In Section \ref{a} we explain why we can assume $a=1$ without loss of generality.
We prove this conjecture by finding an exact expression for the number
of  $\mathbb F_{p^n}$-rational points on $C_k^{(p)}=C_k^{(p,1)}$, 
for any $n$, see Section \ref{sectck}.
This is done by first 
finding an exact expression for the number
of  $\mathbb F_{p^n}$-rational points on related curves
$B_k^{(p)}$, see Section \ref{sectbk}.
Sections \ref{sectb0} and \ref{sectc0} deal with the curves $B_0^{(p)}$
and $C_0^{(p)}$ respectively, which need separate consideration.
Section \ref{proofconj} gives the proof of Conjecture \ref{conj}.
In Section \ref{Bkmaps} we consider the corresponding divisibility 
result for the $B_k$ family.
Section \ref{opp} contains some results on the opposite problem to the conjecture;
we prove that if $k$ does not divide $\ell$ then 
the L-polynomial of $C_k^{(p)}$ does not divide the L-polynomial of $C_\ell^{(p)}$.

We will usually drop the superscript in $C_k^{(p)}=C_k^{(p,1)}$, and write $C_k$.
The trace map is always the absolute trace, unless otherwise stated.
Throughout the paper $\left(\frac{\cdot }{p}\right)$ denotes the Legendre symbol.
Finally, the L-polynomials of a more general class of curves
than $C_k$ were found in \cite{Betal}, however they were L-polynomials
over an extension of $\F_p$, and not L-polynomials over $\F_p$, which is the
subject of this paper.

\section{Background}
In this section we will give some basic facts that we will use. 
Some of this requires that $p$ is odd, some of it does not, 
but we remind that reader that we are going to assume $p$ is
odd for this paper.

\subsection{More on Curves}\label{morec}

Let $p$ be a prime and let $q=p^r$ where $r$ is a positive integer.
Let $X$ be a projective smooth absolutely irreducible curve of genus $g$ 
defined over $\mathbb{F}_q$. Let $\eta_1,\cdots,\eta_{2g}$ be the roots of the 
reciprocal of the $L$-polynomial of $X$ over 
$\mathbb F_{q}$ (sometimes called the Weil numbers of $X$, or
Frobenius eigenvalues). Then, for any $n\geq 1$, 
the number of rational points of $X$ over $\mathbb F_{q^{n}}$ is
given by
\begin{equation}\label{eqn-sum of roots}
\#X(\mathbb F_{q^{n}})=(q^{n}+1)- \sum\limits_{i=1}^{2g}\eta_i^n. 
\end{equation}
The Riemann Hypothesis for curves over finite fields  
states that $|\eta_i|=\sqrt{q}$ for all $i=1,\ldots,2g$. 
It follows immediately from this property and \eqref{eqn-sum of roots}
that
\begin{equation}
|\#X(\mathbb F_{q^n})-(q^n+1)|\leq 2g\sqrt{q^n}
\end{equation}
which is the Hasse-Weil bound.

We call $X(\mathbb F_{q})$ \emph{maximal} if $\eta_i=-\sqrt{q}$ for all $i=1,\cdots,2g$, so the Hasse-Weil upper bound is met. Equivalently, $X(\mathbb F_{q})$ is maximal
if and only if $L_X(T)=(1+\sqrt{q} T)^{2g}$.

We call $X(\mathbb F_{q})$ \emph{minimal} if $\eta_i=\sqrt{q}$ for all $i=1,\cdots,2g$,
so the Hasse-Weil lower bound is met.
Equivalently, $X(\mathbb F_{q})$ is minimal
if and only if $L_X(T)=(1-\sqrt{q} T)^{2g}$.

Note that if $X(\mathbb F_{q})$ is minimal or maximal then $q$ must be a square 
(i.e.\ $r$ must be even).

The following properties follow immediately.

\begin{prop} \label{minimal-prop}
	\begin{enumerate}
		\item If $X(\mathbb F_{q})$ is maximal then  $X(\mathbb F_{q^{n}})$ is minimal for even $n$ and maximal for odd $n$. 
		\item  If $X(\mathbb F_{q})$ is minimal then  $X(\mathbb F_{q^{n}})$ is minimal for all $n$.
	\end{enumerate}
\end{prop}

We also record another Proposition here. 

\begin{prop}\label{pureimag}
If $X$ is a curve defined over $\F_q$ and $X(\mathbb F_{q^{2n}})$ is maximal,
then $\#X(\mathbb F_{q^n})=q^n+1$ and the L-polynomial of $X$ over $\F_{q^n}$
is $(1+q^{n}t^2)^{g}$.
\end{prop}

\begin{proof}
Let $\eta_1,\cdots,\eta_{2g}$ be the Weil numbers of $X$ over $\F_q$.
Then $\eta_j^{2n}=-\sqrt{q^{2n}}=-q^n$ for all $j$ because $X(\mathbb F_{q^{2n}})$ is maximal.
But then $\eta_j^{n}=\pm i \sqrt{q^n}$ for all $j$, which implies that 
\[
\sum\limits_{j=1}^{2g}\eta_j^n =(q^{n}+1)-\#X(\mathbb F_{q^{n}})
\]
is a purely imaginary complex number and also an integer.
This number is therefore 0.
\end{proof}

%\begin{proof}
%	Since $\eta_i^n=\pm \sqrt{q^{n}}$ for all $i=1,\cdots,2g$, we have  $\eta_i^{2n}=(\pm \sqrt{q^{n}})^2=\sqrt{q^{2n}}$ for all $i=1,\cdots,2g$. The proof follows by definition of a minimal curve.
%\end{proof}

\subsection{Supersingular Curves}\label{supsing}

A curve $X$ of genus $g$ defined over $\mathbb F_q$ 
($q=p^r$) is \emph{supersingular} if any of the following 
equivalent properties hold.

\begin{enumerate}
	\item All Weil numbers of $X$ have the form $\eta_i = \sqrt{q}\cdot \zeta_i$ where $\zeta_i$ is a root of unity.
	\item The Newton polygon of $X$ is a straight line of slope $1/2$.
	\item The Jacobian of $X$ is geometrically isogenous to $E^g$ where
	$E$ is a supersingular elliptic curve.
	\item If $X$ has $L$-polynomial
	$L_X(T)=1+\sum\limits_{i=1}^{2g} c_iT^i$ 
	then
	$$ord_p(c_i)\geq \frac{ir}{2}, \ \mbox{for all $i=1,\ldots
		,2g$.}$$
\end{enumerate}

By the first property, a supersingular curve defined over $\mathbb F_q$ becomes minimal over some finite extension of $\mathbb F_q$.
Conversely, any minimal or maximal curve is supersingular.

\subsection{Quadratic forms}\label{QF}

We now recall the basic theory of quadratic forms over $\mathbb{F}_{q}$, where $q$ is odd.

Let $K=\mathbb{F}_{q^n}$, and 
let $Q:K\longrightarrow \mathbb{F}_{q}$ be a quadratic form.
The polarization of $Q$ is the symplectic bilinear form $B$ defined by $B(x,y)=Q(x+y)-Q(x)-Q(y)$. By definition the radical of $B$ (denoted $W$) is $
W =\{ x\in K : B(x,y)=0 \text{  for all $y\in K$}\}$. The rank of $B$ is defined to be $n-\dim(W)$. The rank of $Q$ is defined to be the rank of $B$.

The following result is well known, see Chapter $6$ of \cite{lidl}  for example.
\bigskip

\begin{prop}\label{counts}
	Continue the above notation. Let $N=|\{x\in K : Q(x)=0\}|$, and let $w=\dim(W)$. If $Q$ has odd rank then $N=q^{n-1}$; if $Q$ has even rank then $N=q^{n-1}\pm (q-1)q^{(n-2+w)/2}$.
\end{prop}

In this paper we will be concerned with quadratic forms of the type
$Q(x)=\Tr(f(x))$ where $f(x)$ has the form $\sum a_{ij} x^{p^i+p^j}$.
If $N$ is the number of $x\in \mathbb{F}_{p^n}$ with $\Tr(f(x))=0$, then
because elements of trace 0 have the form $y^p-y$,  finding $N$ is equivalent
to finding the exact number of $\mathbb{F}_{p^n}$-rational points on 
the curve $C: y^p-y=f(x)$.
Indeed, 
\begin{equation}\label{quadpts}
\#C(\mathbb F_{p^n})=pN+1.
\end{equation}

\subsection{Discrete Fourier Transform}

In this section we recall the statement of  the Discrete Fourier Transform
and its inverse.

\begin{prop}[Inverse Discrete Fourier Transform]
	Let $N$ be a positive integer and let $w_N$ be a primitive $N$-th root of unity
	over any field where $N$ is invertible.  If $$F_n=\sum_{j=0}^{N-1}f_jw_N^{-jn}$$ for $n=0,1\cdots, N-1$ then we have  $$f_n=\frac1{N}\sum_{j=0}^{N-1}F_jw_N^{jn}$$ for $n=0,1\cdots, N-1$.
\end{prop}

\subsection{Relations on the Number of Rational Points}

In this section we state a theorem which allows us to find the number of 
$\F_{p^n}$-rational points of a supersingular curve by finding the the number of 
$\F_{p^m}$-rational points only for the divisors $m$ of $s$, where the Weil numbers 
are $\sqrt{p}$ times an $s$-th root of unity. Note that $s$ is even because
equality holds in the Hasse--Weil bound over $\F_{p^s}$.

\begin{thm}[\cite{MY2}]\label{reduction-thm} 	
	Let $p$ be an odd prime.
	Let $X$ be a supersingular curve of genus $g$ defined over $\mathbb F_p$ 
	whose Weil numbers are $\sqrt{p}$ times an $s$-th root of unity. 
	Let $n$ be a positive integer, let $\gcd (n,s)=m$ and write $n=mt$. Then we have\bigskip \\
	$-p^{-n/2}[\#X(\F_{p^n})-(p^n+1)]=$
	$$\begin{cases}
	-p^{-m/2}[\#X(\F_{p^m})-(p^m+1)] &\text{if } m \text{ is even},\\
		-p^{-m/2}[\#X(F_{q^m})-(q^m+1)]&\text{if } m \text{ is odd and } p\mid t,\\
	-p^{-m/2}[\#X(\F_{p^m})-(p^m+1)]\left(\frac{(-1)^{(t-1)/2}t}{p}\right)&\text{if } 
	m \text{ is odd and } p\nmid t,
	\end{cases}$$  where $\left(\frac{.}{p}\right)$ is the Legendre symbol.
\end{thm}

\subsection{A Divisibilty Theorem}

The following theorem is well-known.

\begin{thm}\label{chapman:KleimanSerre} 
(Kleiman--Serre)
If there is a surjective morphism of curves 
$C \longrightarrow D$ that is defined over $\mathbb{F}_q$
then $\mathrm{L}_{D}(T)$ divides $\mathrm{L}_C(T)$.
\end{thm}

This theorem is sometimes used to show divisibility. 
The $p=2$ case of Conjecture \ref{conj} was proved in \cite{conj4}
by finding a map $C_{km}^{(2)} \longrightarrow C_{k}^{(2)}$.
However, there are
cases where there is no map of curves and yet there is divisibility of
L-polynomials. 
We suspect that $C_k^{(p)} $ and $C_{2k}^{(p)} $ is such a case, see Theorem \ref{divisibility-2k}.
We are unable to find a map $C_{2k}^{(p)} \longrightarrow C_{k}^{(p)}$
when $p>2$.

\section{The Curve $B_0: y^p-y=x^2$ over $\mathbb F_p$}\label{sectb0}

From now on in this paper we will assume that $p$ is an odd prime.

Given a bilinear form $B$ we define
$$W^{(n)}:=\{x\in \mathbb F_{p^n} \: | \: B(x,y)=0 \text{ for all } y \in \mathbb F_{p^n} \}.$$

In this section we  will give the exact number of $\mathbb{F}_{p^n}$-rational points on 
$B_0:y^p-y=x^{2}$ for all positive integers $n$.
Note that $B_0$ has genus $(p-1)/2$.

\begin{lemma}\label{b01}
	The number of $\mathbb F_p$-rational points of $B_0$ is $p+1$.
\end{lemma}
\begin{proof} 
	Since $x^2=0$ if and only if $x=0$ and since $y^p-y=0$ for all $y \in \mathbb F_p$, we have that the number of $\mathbb F_p$-rational points of $B_0$
	(including $\infty$)  is $p+1$.
\end{proof}

\begin{lemma}\label{radical-lemma}
	Let $n$ be a positive integer. The radical of the quadratic form $Q_0(x)=\Tr(x^2)$ is $\{0\} $where $\Tr: \mathbb F_{p^n} \to \mathbb F_p$ is the trace map.
\end{lemma}
\begin{proof}
We have $$B_0(x,y):=Q_0(x+y)-Q_0(x)-Q_0(y)=\Tr(2xy)$$ and 
$W^{(n)}=\{0\}$
because $\Tr (xy)$ is a non-degenerate bilinear form. 
\end{proof}

\begin{lemma}\label{b02}
	The number of $\mathbb F_{p^2}$-rational points of $B_0$ is $$\begin{cases}
	p^2+1-(p-1)p &\text{ if } p \equiv 1 \mod 4,\\
	p^2+1+(p-1)p &\text{ if } p \equiv 3 \mod 4.\\
	\end{cases}$$
\end{lemma}
\begin{proof} 
	Since $2-\dim(W^{(2)})=2-0=2$ is even by Lemma \ref{radical-lemma}, 
	the $N$ in Proposition \ref{counts} is equal to $p \pm (p-1)$.
	By \eqref{quadpts} we get 
	\[
	\#B_0(\mathbb F_{p^2})=pN+1=p^2+1\pm p(p-1).
	\]
	Because the genus of $B_0$ is $(p-1)/2$ we have $2g\sqrt{p^2}=p(p-1)$ and so
	the curve $B_0$ is maximal or minimal over $\mathbb F_{p^2}$
	because the Hasse-Weil bound is met.
	
	Let $\Tr:\mathbb F_{p^2}\to \mathbb F_p$ be the trace map. Then $$Tr(x^2)=x^2+x^{2p}=x^2(x^{2p-2}+1).$$
	
	We know that $B_0$ is maximal or minimal over $\mathbb F_{p^2}$. If it is minimal (resp. maximal), then $$|\{ x \in \mathbb F_{q^2} \: | \: Tr(x^2)=0\}|=1 \text{ (resp. $2p-1$)}.$$
	
	In other words, the degree of the greatest common divisor $(x^{2p}+x^2, x^{p^2}-x)$ is $$\begin{cases}
	1  &\text{ if } B_0 \text{ is minimal over } \mathbb F_{q^2},\\
	2p-1  &\text{ if } B_0 \text{ is maximal over } \mathbb F_{q^2}\\
	\end{cases}$$ or the degree of the greatest common divisor $(x^{2p-2}+1, x^{p^2-1}-1)$ is $$\begin{cases}
	0  &\text{ if } B_0 \text{ is minimal over } \mathbb F_{q^2},\\
	2p-2  &\text{ if } B_0 \text{ is maximal over } \mathbb F_{q^2}.\\
	\end{cases}$$ 
	
	Assume $p \equiv 1 \mod 4$. Then  $(p+1)/2$ is a positive odd integer and $$x^{p^2-1}-1=(x^{2p-2})^{(p+1)/2}-1\equiv (-1)^{(p+1)/2}-1=-2 \mod (x^{2p-2}+1)$$ which implies that $x^{2p-2}+1$ does not divide $x^{p^2-1}-1$.  Therefore, $$(x^{2p-2}+1, x^{p^2-1}-1)=1.$$
	
	Assume $p \equiv 3 \mod 4$. Then we have that $(p+1)/4$ is a positive integer and $$x^{p^2-1}-1=(x^{4p-4})^{(p+1)/4}-1$$ is divisible by $x^{4p-4}-1$ which equals to $(x^{2p-2}+1)(x^{2p-2}-1)$. Hence $x^{p^2-1}-1$ is divisible by $x^{2p-2}+1$. Therefore, $$(x^{2p-2}+1, x^{p^2-1}-1)=x^{2p-2}+1.$$
\end{proof}

\begin{thm}\label{thm-B0}
	Let $p\equiv 1 \mod 4$ and $n\ge 1$ be an integer. Then  
	$$-p^{-n/2}\left[\#B_0(\mathbb F_{p^n})-(p^n+1)\right]=\begin{cases} 0 &\text{ if $n$ is odd},\\p-1 &\text{ if $n$ is even}.  \end{cases}$$	Let $p\equiv 3 \mod 4$ and $n\ge 1$ be an integer. Then 
	$$-p^{-n/2}\left[\#B_0(\mathbb F_{p^n})-(p^n+1)\right]=\begin{cases} 0 &\text{ if } (4,n)=1,\\-(p-1) &\text{ if } (4,n)=2,\\p-1 &\text{ if } (4,n)=4.  \end{cases}$$
\end{thm}
\begin{proof}
	It follows by Lemma \ref{b01}, \ref{b02} and Theorem \ref{reduction-thm}.
\end{proof}

\section{The Curve $C_0: y^p-y=x^2+x$ over $\mathbb F_p$}\label{sectc0}

In this section we will give the exact number of $\mathbb{F}_{p^n}$-rational points on $C_0: y^p-y=x^2+x$  for all positive integer $n$.

Let $n \ge 1$ be a positive integer. The map $(x, y)\to (x-2^{-1},y)$ is a one-to-one map over $\mathbb F_{p^n}^2$. Let $\Tr:\mathbb F_{p^n}\to \mathbb F_p$ be the trace map.  Since $$\Tr\left( (x-2^{-1})^2+x\right)=\Tr\left( x^2+4^{-1}\right)=\Tr(x^2)+n4^{-1},$$ we can use the information on the curve $B_0$.
\begin{lemma}\label{C0p}
	Let $n$ be a positive integer. The number of $\mathbb F_{p^{pn}}$-rational points of $C_0$ equals  the number of $\mathbb F_{p^{pn}}$-rational points of $B_0$.
\end{lemma}
\begin{proof} 
	Let $\Tr:\mathbb F_{p^{pn}}\to \mathbb F_p$ be the trace map. Since $p\cdot n$ is divisible by $p$, we have $$Tr(x^2+x)=Tr(x^2).$$ Hence we have the result.
\end{proof}

\begin{lemma}\label{C0B0np1}
	Let $n$ be a positive integer with $(n,p)=1$. If $\#B_0(\mathbb F_{p^n})-(p^n+1)\not= 0$
	then
	$$-(p-1)\bigg(\#C_0(\mathbb F_{p^n})-(p^n+1)\bigg)=
	\bigg( \#B_0(\mathbb F_{p^n})-(p^n+1)\bigg).$$
\end{lemma}
\begin{proof}
	The proof of this lemma is exactly the same as that of Lemma \ref{evenspread}.
	\end{proof}

\begin{lemma}\label{C01}
	The number $\#C_0(\mathbb F_p)$ is $2p+1$.
\end{lemma}
\begin{proof}
	We have  $y^p-y=0$ for all $y \in \mathbb F_p$. Also
	$x^2+x=x(x+1)=0$ if and only if $x=0$ or $x=-1$.
	 Therefore, $\#C_0(\mathbb F_p)=2\cdot p+1$.
\end{proof}

\begin{lemma}\label{C0np1}
	 $C_0(\mathbb F_{p^{2p}})$ is minimal   if $p\equiv 1 \mod 4$ and maximal $p\equiv 3 \mod 4$.
\end{lemma}
\begin{proof} 
	By Lemma \ref{C0p} we know that  $\#C_0(\mathbb F_{p^{2p}})=\#B_0(\mathbb F_{p^{2p}})$. Hence it follows by Theorem \ref{thm-B0}.
\end{proof}

We put all these results together in the final Theorem of this section.

\begin{thm}
	Let $p\equiv 1 \mod 4$ and $n\ge 1$ be an integer. Then we have that$$-p^{-n/2}\left[\#C_0(\mathbb F_{p^n})-(p^n+1)\right]=\begin{cases} -\left(\frac{n}{p}\right) \sqrt{p}&\text{ if } (n,2p)=1,\\ -1 &\text{ if } (n,2p)=2,\\ 0 &\text{ if } (n,2p)=p,\\ p-1 &\text{ if } (n,2p)=2p.\end{cases}$$
	Let $p\equiv 3 \mod 4$ and $n\ge 1$ be an integer. Then we have that $$-p^{-n/2}\left[\#C_0(\mathbb F_{p^n})-(p^n+1)\right]=\begin{cases} -\left(\frac{(-1)^{(n-1)/2}n}{p}\right) \sqrt{p}&\text{ if } (n,4p)=1,\\ 1 &\text{ if } (n,4p)=2,\\ -1 &\text{ if } (n,4p)=4,\\ 0 &\text{ if } (n,4p)=p,\\ -(p-1) &\text{ if } (n,4p)=2p,\\p-1 &\text{ if } (n,4p)=4p.\end{cases}$$	
\end{thm}
\begin{proof}
It follows by Lemmas \ref{C0p}, \ref{C0B0np1}, \ref{C01}, \ref{C0np1} and Theorems \ref{thm-B0} and \ref{reduction-thm}.
\end{proof}

\section{The Curve $B_k: y^p-y=x^{p^k+1}$ over $\mathbb F_p$}\label{sectbk}

In this section we  will give the exact number of $\mathbb{F}_{p^n}$-rational points on 
$B_k^{(p)}=B_k: y^p-y=x^{p^k+1}$  for all positive integers $k$ and $n$.

\begin{lemma}
	Let $d\mid k$. The number of $\mathbb F_{p^d}$-rational points of $B_k$ is 
	equal to the number of $\mathbb F_{p^d}$-rational points of $B_0$.
\end{lemma}
\begin{proof}
	Since $x^{p^k+1}=x^{p^k}\cdot x=x\cdot x=x^2$ in $\mathbb F_{p^d}$ for all $d\mid k$, the result is immediate.
\end{proof}

\begin{lemma}
	Let $d\mid k$ with $2d \nmid k$. The number of $\mathbb F_{p^{2d}}$-rational points of $B_k$ is $(p^{2d}+1)-(p-1)p^{d}$.
\end{lemma}
\begin{proof}
	Since $d\mid k$ and $2d\nmid k$, we have $e:= k/d$ is odd. Define 
	$\Tr_n: \mathbb F_{p^n} \to \F_p$ be the trace map. 
	We have \begin{align*}
	\Tr_{2d}(x^{p^k+1})&=\Tr_d(x^{p^k+1}+x^{p^{k+d}+p^d})\\  &=\Tr_d(x^{p^k+1}+x^{p^{ed+d}+p^d})\\&=\Tr_d(x^{p^{ed}+1}+x^{p^{d(e+1)}+p^d})\\&=\Tr_d(x^{p^d+1}+x^{p^{d(e+1)}+p^d})\\&=\Tr_d(x^{p^d+1}+x^{1+p^d})\\&=\Tr_d(2x^{p^d+1}).
	\end{align*}
	Since $x\to x^{p^d+1}$ is $p^d+1$-to-$1$ map from $\mathbb F_{p^{2d}}^\times$ to $\mathbb F_{p^d}^\times$ and since $\Tr_d(x)$ is a linear map from $\mathbb F_{p^{d}}$ to $\mathbb F_p$, we have that the number of $\mathbb F_{p^{2d}}$-rational points of $B_k$ is $$1+p(1+(p^d+1)(p^{d-1}-1))=(p^{2d}+1)-(p-1)p^{d}.$$
\end{proof}

\begin{lemma}
	The curve $B_k$ is minimal over $\mathbb F_{p^{4k}}$. 
\end{lemma}

\begin{proof}
	Define $Q_k(x)=\Tr(x^{p^k+1})$ where $\Tr: \mathbb F_{p^n} \to \mathbb F_p$ is the trace map ($n=4k$). 
	We have $$B(x,y):=Q(x+y)-Q(x)-Q(y)=\Tr(x^{p^k}y+xy^{p^k})=\Tr(y^{p^{k}}(x^{p^{2k}}+x))$$ and $$W^{(n)}:=\{x\in \mathbb F_{p^n} \: | \: B(x,y)=0 \text{ for all } y \in \mathbb F_{p^n} \}=\{x\in \mathbb F_{p^n} \: | \:x^{p^{2k}}+x=0 \}.$$ 
	So $W^{(4k)} \subseteq \F_{p^{4k}}$ and so the rank of $Q_k$ is 
	$n-\dim W^{(n)}=4k-2k$ which is even, and so the $N$ in Proposition \ref{counts} is equal to $p^{n-1} \pm (p-1)p^{3k-1}$.
	By \eqref{quadpts} we get 
	\[
	\#B_0(\mathbb F_{p^{4k}})=pN+1=p^n+1\pm (p-1)p^{3k}.
	\]
	The genus of $B_k$ is $p^k(p-1)/2$ so  $2g\sqrt{p^{4k}}=p^{3k}(p-1)$, and so
	the curve $B_k$ is maximal or minimal over $\mathbb F_{p^{4k}}$
	because the Hasse-Weil bound is met.

	If  the curve $B_k$ is maximal over $\mathbb F_{p^{4k}}$, then  $\#B_k(\mathbb F_{p^{2k}})$ has to be $p^{2k}+1$ by Proposition \ref{pureimag}. 
	However, $W^{(2k)}=\{0\}$ and so $2k-\dim W^{(2k)}=2k-0=2k$ is even, which means that 
	$\#B_k(\mathbb F_{p^{2k}})$ cannot be $p^{2k}+1$ by Proposition \ref{counts}.
	Hence  the curve $B_k$ is minimal over $\mathbb F_{p^{4k}}$.
\end{proof}

\begin{cor}\label{bkperiod}
 We have $\zeta^{4k}=1$ for all $\zeta$ where $\sqrt{q} \zeta$ is a
Weil number of $B_k$.
\end{cor}

\begin{proof}
We have shown that $B_k$ is minimal over $\mathbb F_{p^{4k}}$,
and it follows from Sections \ref{morec} and \ref{supsing}.
\end{proof}

\begin{lemma}
	Let $d\mid 4k$ with $d \nmid 2k$. The number of $\mathbb F_{p^{d}}$-rational points of $B_k$ is $(p^d+1)-(p-1)p^{3d/4}$.
\end{lemma}
\begin{proof}
	Since $d\mid 4k$ and $d\nmid 2k$, we have $e:= 4k/d$ is odd and $d=4f$ for some integer $f$. Define $\Tr_n: \mathbb F_{p^n} \to F_p$ be the trace map. 
	We have \begin{align*}
	\Tr_{d}(x^{p^k+1})&=\Tr_{f}(x^{p^k+1}+x^{p^{k+f}+p^f}+x^{p^{k+2f}+p^{2f}}+x^{p^{k+3f}+p^{3f}})\\  &=\Tr_{f}(x^{p^{ef}+1}+x^{p^{f(e+1)}+p^f}+x^{p^{f(e+2)}+p^{2f}}+x^{p^{f(e+3)}+p^{3f}})\\&=\Tr_{f}(x^{p^f+1}+x^{p^{2f}+p^f}+x^{p^{3f}+p^{2f}}++x^{p^{4f}+p^{3f}})\\&=\Tr_d(x^{p^f+1}).
	\end{align*}
	Since $4f=d \mid d$, the curve $B_f$ is minimal over $\mathbb F_{p^d}$ and hence  The number of $\mathbb F_{p^{d}}$-rational points of $B_k$ is $$(p^d+1)-(p-1)p^{3d/4}.$$
\end{proof}

\begin{cor}\label{cor-bk}
	If $d \mid k$ and $d$ is odd, then $$\#B_k(\mathbb F_{p^d})=p^d+1.$$	If $d \mid k$ and $d$ is even, then $$\#B_k(\mathbb F_{p^d})=\begin{cases}(p^d+1)+(p-1)p^{d/2} &\text{ if } 2 \mid\mid d \text{ and } p\equiv 3 \mod 4, \\
	(p^d+1)-(p-1)p^{d/2} &\text{ if } 4 \mid d \text{ or } p\equiv 1 \mod 4.\end{cases}$$ If $d\nmid k$ and $\frac d2\mid k$, then  $$\#B_k(\mathbb F_{p^d})=p^d+1-(p-1)p^{d/2}.$$ If $d\nmid 2k$ and $d \mid 4k$, then $$\#B_k(\mathbb F_{p^d})=(p^d+1)-(p-1)p^{3d/4}.$$
\end{cor}

We put all these results together.

\begin{thm}
	Let $n\ge 1$ be an integer and let $d=(n,4k)$. Then we have that$$-p^{-n/2}\left[\#B_k(\mathbb F_{p^n})-(p^n+1)\right]=\begin{cases} 0&\text{ if } d\mid k \text{ and $d$ is odd},\\ (-1)^{n(p-1)/4}(p-1) &\text{ if }  d\mid k \text{ and $d$ is even},\\ -(p-1)&\text{ if }  d\nmid k \text{ and } \frac d2\mid k,\\ (p-1)p^{(k,n)} &\text{ if }  d\nmid 2k \text{ and } d\mid 4k.\end{cases}$$
\end{thm}
\begin{proof}
	It follows by Corollary \ref{cor-bk} and Theorem \ref{reduction-thm}.
\end{proof}

\section{The Curve $C_k: y^p-y=x^{p^k+1}+x$ over $\mathbb F_p$}\label{sectck}

In this section we  will give the exact number of $\mathbb{F}_{p^n}$-rational points on 
$C_k: y^p-y=x^{p^k+1}+x$  for all positive integers $k$ and $n$.

Let $(x, y)\to (x-2^{-1},y)$ is a one-to-one map over $\mathbb F_{p^n}^2$. Since 
\begin{align}
\Tr\left( (x-2^{-1})^{p^k+1}+(x-2^{-1})\right)&=
\Tr\left( x^{p^k+1}-2^{-1}x^{p^k}+2^{-1}x-4^{-1}\right)\\
&=\Tr(x^{p^k+1})-n4^{-1}, \label{onetoone}
\end{align}
we can use the information on the curve $B_k$.

\begin{lemma}\label{isooverp}
	If $p|n$, the number of $\mathbb F_{p^{n}}$-rational points of $C_k$ equals  the 
	number of $\mathbb F_{p^{n}}$-rational points of $B_k$.
\end{lemma}
\begin{proof} 
	Since $n$ is divisible by $p$, 
by  \eqref{onetoone} 
we have $$\mid \{ x\in \mathbb F_{p^{n}} : Tr(x^{p^k+1}+x)=0 \} \mid =
\mid \{ x\in \mathbb F_{p^{n}} : Tr(x^{p^k+1})=0 \}.$$ Hence we have the result.
\end{proof}

\begin{cor}\label{ckperiod1}
 If $p | k$ we have $\zeta^{4k}=1$ for all $\zeta$ where $\sqrt{q} \zeta$ is a
Weil number of $C_k$.
\end{cor}

\begin{proof}
Follows from  Lemma \ref{isooverp} and Corollary \ref{bkperiod}.
\end{proof}

{\bf Remark}. 
It follows from Lemma \ref{isooverp}  that $C_k$ and $B_k$ have the same
L-polynomial when considered as curves defined over $\F_{p^p}$.
They do not have the same L-polynomial when considered as curves defined over $\F_{p}$,
as the results in this paper show (see Lemma \ref{evenspread} below).
This means that the $p$-th powers of the Weil numbers of $B_k$ and $C_k$ 
(considered as curves defined over $\F_{p}$) are equal,
but the Weil numbers themselves are not the same.
For example, the L-polynomial of $B_2^{(3)}$ is
\[
19683T^{18} + 6561T^{16} - 486T^{10} - 162T^8 + 3T^2 + 1
\]
and the L-polynomial of $C_2^{(3)}$ is
\[
19683T^{18} - 19683T^{17} + 6561T^{16} + 243T^{10} - 243T^9 + 81T^8 + 3T^2 - 3T+ 1.
\]
For both of these, the polynomial whose roots are the cubes of the roots is
\[
1 + 27T^2 - 1062882T^8 - 28697814T^{10} + 282429536481T^{16} +
    7625597484987T^{18}
\]
which is the L-polynomial of both $B_2$ and $C_2$  considered as curves defined over $\F_{3^3}$.

\begin{lemma}
	Let $d\mid k$. The number of $\mathbb F_{p^d}$-rational points of $C_k$ is the number of $\mathbb F_{p^d}$-rational points of $C_0$.
\end{lemma}
\begin{proof}
	Since $$x^{p^k+1}+x=x^{p^k}\cdot x+x=x\cdot x+x=x^2+x$$ in $\mathbb F_{p^d}$ for all $d\mid k$, the result is immediate.
\end{proof}

\begin{lemma}\label{evenspread}
	Let $n$ be a positive integer with $(n,p)=1$. If $\#B_k(\mathbb F_{p^n})-(p^n+1)\not= 0$
	then
	$$-(p-1)\bigg(\#C_k(\mathbb F_{p^n})-(p^n+1)\bigg)=
	\bigg( \#B_k(\mathbb F_{p^n})-(p^n+1)\bigg).$$
\end{lemma}

\begin{proof}
	Let $b:=n4^{-1}\ne 0$, let
	\[
N_0=| \{ x \in \F_{p^n} : \Tr(x^{p^k+1})=0 \} |
\]
and let
\[
N_1=| \{ x \in \F_{p^n} : \Tr(x^{p^k+1})=b \} |.
\]
By \eqref{onetoone} and also \eqref{quadpts}
we get $\#C_k(\mathbb F_{p^n}) = pN_1+1$.
So 
\begin{equation}\label{ckn1}
\#C_k(\mathbb F_{p^n})-(p^n+1)=pN_1-p^n.
\end{equation}
The nonzero values of the quadratic form $\Tr(x^{p^k+1})$ are evenly distributed over the nonzero elements of $\F_p$ by \cite[Theorem 6.26]{lidl},
	so \[ N_0+(p-1)N_1=p^n. \]
Substituting for $N_1$ into \eqref{ckn1} gives
\[
\#C_k(\mathbb F_{p^n})-(p^n+1)=\frac{p^{n+1}-pN_0}{p-1}-p^n
\]
or
\[
(p-1)\bigg( \#C_k(\mathbb F_{p^n})-(p^n+1)\bigg)
=p^{n}-pN_0.
\]
Finally, by \eqref{quadpts} again we note that
\[
\#B_k(\mathbb F_{p^n})-(p^n+1)=pN_0-p^n.
\]	
\end{proof}

\begin{cor}\label{ckperiod2}
 If $p \nmid k$ we have $\zeta^{4kp}=1$ for all $\zeta$ where $\sqrt{q} \zeta$ is a
Weil number of $C_k$.
\end{cor}

\begin{proof}
The previous lemma shows that that $C_k$ is minimal over $\mathbb F_{p^{4kp}}$,
and it follows from Sections \ref{morec} and \ref{supsing}.
\end{proof}

\begin{cor}\label{cor-Ck}
	 Let $k$ be a positive integer. Define 
	$$l=\begin{cases}
	k &\text{if }p\mid k,\\
	kp &\text{if }p\nmid k.
	\end{cases}$$ \\ If $d \mid l$ and $d$ is odd and relatively prime to $p$, then 
	$$\#C(\mathbb F_{p^d})=p^d+1+\left(\frac{(-1)^{(d-1)/2}d}{p}\right)p^{(d+1)/2}.$$
	If $d \mid l $ and $d$ is odd and divisible by $p$, then $$\#C_k(\mathbb F_{p^d})=0.$$	If $d \mid l $ and $d$ is even and relatively prime to $p$, then $$\#C_k(\mathbb F_{p^d})=\begin{cases}(p^d+1)-p^{d/2} &\text{ if } 2 \mid\mid d \text{ and } p\equiv 3 \mod 4, \\
	(p^d+1)+p^{d/2} &\text{ if } 4 \mid d \text{ or } p\equiv 1 \mod 4.\end{cases}$$ If $d \mid l $ and $d$ is even and divisible by $p$, then $$\#C_k(\mathbb F_{p^d})=\begin{cases}(p^d+1)+(p-1)p^{d/2} &\text{ if } 2 \mid\mid d \text{ and } p\equiv 3 \mod 4, \\
	(p^d+1)-(p-1)p^{d/2} &\text{ if } 4 \mid d \text{ or } p\equiv 1 \mod 4.\end{cases}$$ If $d\nmid l$ and $\frac{d}{2} \mid l$ and $d$ is relatively prime to $p$, then $$C_k(\mathbb F_{p^d})=(p^d+1)+p^{d/2}.$$ If $d\nmid l$ and $\frac{d}{2} \mid l$ and $d$ is relatively prime to $p$, then $$C_k(\mathbb F_{p^d})=(p^d+1)-(p-1)p^{d/2}.$$ If $d\nmid 2l$, $d \mid 4l$ and $d$ is relatively prime to $p$, then $$\#C_k(\mathbb F_{p^d})=(p^d+1)+p^{3d/4}.$$ If $d\nmid 2l$, $d \mid 4l$ and $d$ is divisible by $p$, then $$\#C_k(\mathbb F_{p^d})=(p^d+1)-(p-1)p^{3d/4}.$$
\end{cor}

\begin{thm}\label{finalck}
		 Let $k$ be a positive integer. Define 
	$$l=\begin{cases}
	k &\text{if }p\mid k,\\
	kp &\text{if }p\nmid k.
	\end{cases}$$
Let $n\ge 1$ be an integer with $d=(n,4l)$. Then we have
	$$-p^{-n/2}[\#C_k(\mathbb F_{p^n})-(p^n+1)]=\begin{cases}
	-\left(\frac{(-1)^{(n-1)/2}n}{p}\right) \sqrt{p}&\text{ if } d \mid l \text{ and $n$ is odd and $p\nmid n$} ,\\ 0 &\text{ if }  d \mid l \text{ and $n$ is odd and $p\mid n$},\\ -(-1)^{n(p-1)/4} &\text{ if }  d \mid l \text{ and $n$ is even and $p\nmid n$},\\ 
	(-1)^{n(p-1)/4}(p-1) &\text{ if }  d \mid l \text{ and $n$ is even and $p\mid n$},\\  -1 &\text{ if } \text{ $d\nmid l$ and $\frac{d}{2} \mid l$ and $p\nmid n$},\\ p-1 &\text{ if } \text{ $d\nmid l$ and $\frac{d}{2} \mid l$ and $p\mid n$}, \\-p^{(k,n)} &\text{ if } d \nmid 2l \text{ and } d \mid 4l \text{ and } p \nmid n,\\(p-1)p^{(k,n)} &\text{ if } d \nmid 2l \text{ and } d \mid 4l \text{ and } p \mid n. 
	\end{cases}$$
\end{thm}
\begin{proof}
	It follows by Corollary \ref{cor-Ck} and Theorem \ref{reduction-thm}.
\end{proof}
	
\section{Divisibility Property of the Curves $C_k$}\label{proofconj}

In this section we will prove  Conjecture \ref{conj}. 
The proof will be broken into a few parts. 
The first part is to show that the L-polynomial of $C_k$
divides the L-polynomial of $C_{2k}$.
The next part is to show that the L-polynomial of $C_k$
divides the L-polynomial of $C_{tk}$ where $t$ is odd.
Finally, these results are combined to prove the conjecture.

\begin{lemma}\label{divisibility-lemma-un} Let $k$ be a positive integer. Define 
$$s=\begin{cases}
	8k &\text{if }p\mid k,\\
	8kp &\text{if }p\nmid k.
	\end{cases}$$ \\
	For $n\geq 1$ define 
	$$U_n=-p^{-n/2}[\#C_{2k}(\mathbb F_{p^n})-\#C_{k}(\mathbb F_{p^n})]$$and 
	write $U_n$ as a linear combination  
	of the $s$-th roots of unity as
	$$U_n=\sum_{j=0}^{s-1}u_jw_{s}^{-jn}.$$ 
	Then we have $$u_n \ge 0$$ for all $n \in \{ 0,1,\cdots, s-1\}$. 
\end{lemma}

\begin{proof} Let $U_0=U_s$.
Write $k=2^vt$ where $v$ is a positive integer and $t$ is an odd integer.  
By Theorem  \ref{finalck} we have $$U_n=\begin{cases}
	0 &\text{if } 2^{v+1} \nmid n, \\
	-1+(-1)^{n(p-1)/4} &\text{if } 2^{v+1} \mid \mid n \text{ and } p \nmid n,\\
	(p-1)(1-(-1)^{n(p-1)/4})) &\text{if } 2^{v+1} \mid \mid n \text{ and } p \mid n,\\
	p^{(k,n)}-1 &\text{if } 2^{v+2} \mid \mid n \text{ and } p \nmid n,\\
	-(p-1)(p^{(k,n)}-1) &\text{if } 2^{v+2} \mid \mid n \text{ and } p \mid n,\\
	-(p^{(2k,n)}-p^{(k,n)}) &\text{if } 2^{v+3} \mid n \text{ and } p \nmid n,\\
	(p-1)(p^{(2k,n)}-p^{(k,n)}) &\text{if } 2^{v+3} \mid n \text{ and } p \mid n.
	\end{cases}$$ 
	
	If $k$ is divisible by $p$, by using  Inverse  Discrete Fourier Transform we have 
	\begin{align*}
	u_{n}&=\frac{1}{8k}\sum_{j=0}^{8k-1}U_jw_{8k}^{jn}\\
	&=\frac{1}{8k}\sum_{j=0}^{4t-1}U_{2^{v+1}j}w_{2t}^{jn}  \quad \textrm{because
	$U_n=0$ if $2^{v+1} \nmid n$}\\
	&\ge \frac{1}{8k}\left(U_0- \sum_{j=1}^{4t-1}|U_{2^{v+2}j}| \right) \quad \textrm{by the
	triangle inequality}\\
	&\ge \frac{1}{8k}\left[(p-1)(p^{2k}-p^{k})-(4t-1)(p-1)p^{k}\right]  \quad \textrm{because
	$U_s=U_0=(p-1)(p^{2k}-p^k)$}\\
	& \qquad \qquad \qquad \qquad \qquad \qquad\qquad \qquad \qquad \qquad \textrm{    and all others are $\le p^k(p-1)$} \\
	&= \frac{1}{8k}(p-1)p^{k}(p^k-4t)\\&\ge \frac{1}{8k}(p-1)p^{k}(p^k-4k)\\ &\ge 0. 
	\end{align*}
	
	If $k$ and $n$ are not divisible by $p$, by using the 
	Inverse  Discrete Fourier Transform we have 
	\begin{align*}
	u_{n}&=\frac{1}{8kp}\sum_{j=0}^{8kp-1}U_jw_{8kp}^{jn}\\
	&=\frac{1}{8kp}\sum_{j=0}^{4tp-1}U_{2^{v+1}j}w_{2tp}^{jn} \\
	&\geq \frac{1}{8kp}\left( U_0+\sum_{j=1}^{p-1}U_{2^{v+3}tj}w_p^j-\sum_{j=0,2t\nmid j}^{2tp-1}|U_{2^{v+2}j}|- \sum_{j=0}^{2tp-1}|U_{2^{v+1}(2j+1)}|\right).
	\end{align*}

 We have $$|U_{2^{v+1}(2j+1)}| \le  \begin{cases}
	2 &\text{ if } p\nmid (2j+1), \\
	2(p-1) &\text{ if } p\mid (2j+1). \\
	\end{cases}$$
	Since there are $2t$ (resp. $2t(p-1)$) integers which is (resp. not) divisible by $p$ between $0$ and $2tp-1$, we have $$\sum_{j=0}^{2tp-1}|U_{2^{v+1}(2j+1)}|\le 2t\cdot 2(p-1)+2t(p-1)\cdot 2=8t(p-1).$$
	Therefore, we have\begin{align*}
	u_n &\ge \frac{1}{8kp}\left[(p-1)(p^{2k}-p^{k})+(p^{2k}-p^{k})-(2tp-p)(p-1)p^{k}-8t(p-1)\right] \\
	&\ge \frac{1}{8k}p^{k}(p^k-(p-1)(2t-1))-\frac{p-1}{p}\\&> -1. 
	\end{align*}
		Since $u_n$ is an integer, we have $u_n\ge 0$.\\

		Assume $k$ is not divisible by $p$ and $n$ is divisible by $p$ and write $n=mp$.  
		We will show that $u_n=0$.
		By the Inverse  Discrete Fourier Transform we have
\begin{align*}
u_{n}&=\frac{1}{8kp}\sum_{j=0}^{8kp-1}U_jw_{8kp}^{jn}\\
&=\frac{1}{8kp}\sum_{j=0}^{4tp-1}U_{2^{v+1}j}w_{4t}^{jm}\\
&=\frac{1}{8kp}\sum_{j=0}^{4t-1}\left[
	\left(\sum_{i=0}^{p-1}U_{2^{v+1}(4ti+j)}\right)w_{2t}^{jm}\right].
	\end{align*}
	 Since $(4t,p)=1$, for any integer $j$ we have 
	 $$\{ 4ti+j \mod p \: |\: 0 \le i\le p-1\}=\{i \mod p\: | \:  0 \le i\le p-1 \}$$
	 and so exactly one of the $4ti+j$ is divisible by $p$.
	 Therefore, for each $j$,
	 \[
	 \sum_{i=0}^{p-1}U_{2^{v+2}(2ti+j)}=0
	 \]
	 because if $j$ is odd then one term is $(p-1)(1-(-1)^{n(p-1)/4})$ and
	 the other $p-1$ terms are $-1+(-1)^{n(p-1)/4}$, if $2\mid \mid j$ then one term is $-(p-1)(p^{(k,n)}-1)$ and
	 the other $p-1$ terms are $p^{(k,n)}-1$,
	 and if $4 \mid j$ is even then one term is $(p-1)(p^{(2k,n)}-p^{(k,n)})$ and
	 the other $p-1$ terms are $-(p^{(2k,n)}-p^{(k,n)})$.
	\end{proof}

We write $L(C_k)$ for $L_{C_k}$.	
	
\begin{cor}\label{divisibility-2k}
	Let $k$ be a positive integer. Then $$L(C_k) \mid L(C_{2k}).$$
\end{cor}
\begin{proof}
	Lemma \ref{divisibility-lemma-un} shows that the multiplicity of each root of $L(C_k)$ is smaller than or equal to 
	its multiplicity as a root of $L(C_{2k})$.
\end{proof}

  \begin{lemma}\label{divisiblity-odd}
	Let $k$ be an integer and $t$ be an odd integer. Then $$L(C_k)\mid L(C_{kt}).$$
 \end{lemma}
 
 \begin{proof}
Let	$X_k: y^p-y=x^{p^k+1}-4^{-1}$ over $\mathbb F_{p}$ and check that
\[
(x,y) \mapsto \left(x-\frac12, y-\frac12\sum_{i=0}^{k-1}x^{p^{i}}\right)
\]
is a map $X_k \longrightarrow C_k$.
The map is defined everywhere and is invertible, so $C_k$ is isomorphic to $X_k$, and hence
 $L(C_k)=L(X_k).$ Therefore, it is enough to show that $L(X_k) \mid L(X_{kt})$. 
 	
 	Since $t$ is odd, $p^k+1$ divides $p^{kt}+1$ and therefore there is a map 
	of curves $X_{kt} \longrightarrow X_k$ given by
	$(x,y)\to (x^{(p^{kt}+1)/(p^k+1)},y)$. Hence 
	$L(X_k) \mid L(X_{kt})$ by Theorem \ref{chapman:KleimanSerre}.
 \end{proof}

\begin{thm}\label{wholeconj}
	Let $k$ and $m$ be positive integers. Then $$L(C_k) \mid L(C_{km}).$$
\end{thm}

\begin{proof}
	If $m=1$, then the result is trivial. Assume $m\ge 2$ and write $m=2^st$ where $t$ is odd. Since $t$ is odd, by Lemma \ref{divisiblity-odd} we have $$L(C_k)\mid L(C_{kt})$$ and by Corollary  \ref{divisibility-2k} we have $$L(C_{2^{i-1}kt})\mid L(C_{2^{i}kt})$$ for all $i\in\{1,\cdots,s\}$. Hence $$L(C_k) \mid L(C_{km}).$$
\end{proof}

\section{Remark on the Divisibility Property of the Curves $B_k$}\label{Bkmaps}

The divisibility property of the curves $B_k$ can be proved in the same way as for $C_k$.
For odd $t$, we have a natural map from $B_{tk}$ to $B_{k}$ which sends $(x,y)$ to 
$(x^{(p^{kt}+1)/(p^k+1)},y)$. 
We are unable to find a map from $B_{2k}$ to $B_k$, so we use a similar argument.
For $n\ge 1$ we can define 
$$U_n=-p^{-n/2}[\#B_{2k}(\mathbb F_{p^n})-\#B_{k}(\mathbb F_{p^n})]=\begin{cases}
0 &\text{if } 2^{v+1} \nmid n, \\
(p-1)(1-(-1)^{n(p-1)/4})) &\text{if } 2^{v+1} \mid \mid n \text{ and } p \mid n,\\
-(p-1)(p^{(k,n)}-1) &\text{if } 2^{v+2} \mid \mid n \text{ and } p \mid n,\\
(p-1)(p^{(2k,n)}-p^{(k,n)}) &\text{if } 2^{v+3} \mid n \text{ and } p \mid n.
\end{cases}$$ and can similarly show that $u_j\ge 0$ for all $j\in\{0,\cdots,4k-1\}$ where 	
$$U_n=\sum_{j=0}^{4k-1}u_jw_{4k}^{-jn}.$$
We write $L(B_k)$ for $L_{B_k}$.

	\begin{thm}
		Let $k$ and $m$ be positive integers. Then $$L(B_k) \mid L(B_{km}).$$
	\end{thm}
The proof is  similar to the proof of Theorem \ref{wholeconj}.

\section{Opposite Direction}\label{opp}

In this section, we will prove that the opposite directions of the divisibility theorems for $B_k$ and $C_k$ are also valid.
To be precise, we have shown that if $k$ divides $\ell$ then the L-polynomial
of $C_k$ (or $B_k$) divides the L-polynomial of $C_\ell$ (or $B_\ell$).
We now prove that if $k$ does not divide $\ell$ then the L-polynomials do not divide.

Let $X$ be a supersingular curve defined over $\F_q$.
The smallest positive integer $s=s_X$ such that $\zeta_i^{s}=1$ for all 
 $i=1,\ldots	,2g$ will be called the \emph{period} of $X$. 
The period depends on $q$, in the sense that $X(\F_{q^n})$
may have a different period to $X(\F_{q})$.

\begin{prop}\label{divisiblity-period}
	Let $C$ and $D$ be supersingular curves over $\mathbb F_q$. If $L(C)$ divides $L(D)$, then $s_C$ divides $s_D$.
\end{prop}

\begin{proof}
	Since $L(C)$ divides $L(D)$, the roots of $L(C)$ are also roots of $L(D)$. Therefore, any Weil number of $C$ is also a Weil number of $D$. Let $\sqrt{q}\zeta$ be a Weil number for $C$. Since $\zeta^{s_D}=1$, the order of $\zeta$ divides  $s_D$. Since this happens for all Weil numbers of $C$, $s_C$ divides $s_D$.  	
\end{proof}

\begin{thm}\label{Bknotdiv}
		Let $k$ and $\ell $ be positive integers such that $k$ does not divide $\ell$. 
		Then $L(B_k)$ does not divide $L(B_{\ell })$.
	\end{thm}

\begin{proof}
By Corollary \ref{bkperiod} 
the period of $B_k$ is $4k$ and the period of $B_\ell$ is $4\ell$. Since $4k$ does not divide $4\ell$, we have $L(B_k)$ does not divide $L(B_\ell)$ by Proposition \ref{divisiblity-period}. 
\end{proof}

\begin{cor}
Let $k$ and $\ell $ be positive integers such that $k<\ell $ and $k$ does not divide $\ell$. 
Then there is no map from $B_\ell$ to $B_k$.
\end{cor}

\begin{proof}
By the Kleiman-Serre theorem (Theorem \ref{chapman:KleimanSerre}) and Theorem \ref{Bknotdiv}.
\end{proof}

Now we turn to $C_k$.

\begin{lemma}\label{Cp-nmid-Cl}
	Let $\ell$ be a positive integer coprime to $p$. Then $L(C_p)$ does not divide $L(C_{\ell})$.
\end{lemma}
\begin{proof}
	We will check the multiplicities of $T-\sqrt{p}$ in $L_{C_p}(T)$ and $L_{C_\ell}(T)$.
	
	Using the Inverse Fourier Transform, the multiplicity of $T-\sqrt{p}$ in $L_{C_p}(T)$ is  $$\frac{1}{4p}\sum_{j=1}^{4p}(p^{-j/2}[\#C_p(\mathbb F_{p^j})-(p^j+1)])\ge\frac{1}{4p} \left[(p-1)p^p-(4p-1)p\right] >0$$ by Theorem \ref{finalck} and triangle inequality.
	
	Using the Inverse Fourier Transform, the multiplicity of $T-\sqrt{p}$ in $L_{C_\ell}(T)$ is  $$\frac{1}{4p\ell}\sum_{j=1}^{4p\ell}(p^{-j/2}[\#C_\ell(\mathbb F_{p^j})-(p^j+1)])$$ 
	which is 0 because $$|\{k \in \mathbb Z \ | \ (4p\ell,k)=d \ \text{ and } \ 1 \le j \le 4p\ell \}|=(p-1)\cdot |\{k \in \mathbb Z \ | \ (4p\ell,k)=dp \  \text{ and } \ 1 \le j \le 4p\ell  \}|$$ for any $d \mid l$ and by Theorem \ref{finalck}.
\end{proof}

\begin{thm}\label{Cknotdiv}
	Let $k$ and $\ell $ be positive integers such that $k$ does not divide $\ell$. 
	Then $L(C_k)$ does not divide $L(C_{\ell })$.
\end{thm}

\begin{proof}
We use Corollaries $\ref{ckperiod1}$ and $\ref{ckperiod2}$ which give the
period of $C_k$.

	Case I: If $p \mid k,l$, then the period of $C_k$ is $4k$ and the period of $C_\ell$ is $4\ell$. Since $4k \nmid 4\ell$,  $L(C_k)$ does not divide $L(C_\ell)$ by Proposition \ref{divisiblity-period}. 
	
	Case II: If $p \nmid k$and  $p\mid \ell$, then the period of $C_k$ is $4kp$ and the period of $C_\ell$ is $4\ell$. Since $4kp \nmid 4\ell$,  $L(C_k)$ does not divide $L(C_\ell)$ by Proposition \ref{divisiblity-period}. 
	
	Case III: If $p \nmid k$and  $p\nmid \ell$, then the period of $C_k$ is $4kp$ and the period of $C_\ell$ is $4\ell p$. Since $4kp \nmid 4\ell p$,  $L(C_k)$ does not divide $L(C_\ell)$ by Proposition \ref{divisiblity-period}.

	Case IV - A: If $p \mid k$,  $p\nmid \ell$ and $(k/p) \nmid \ell$, then the period of $C_k$ is $4k$ and the period of $C_\ell$ is $4\ell p$. Since $4k \nmid 4\ell p$,  $L(C_k)$ does not divide $L(C_\ell)$ by Proposition \ref{divisiblity-period}. 
	
	Case IV - B: If $p \mid k$,  $p\nmid \ell$ and $(k/p) \mid \ell$, then the period of $C_k$ is $4k$ and the period of $C_\ell$ is $4\ell p$. Since $4k \mid 4\ell p$, we cannot use the Proposition \ref{divisiblity-period}. 
	
	Since $p\mid k$, we have $L(C_p)$ divides $L(C_k)$. If $L(C_k)$ divides $L(C_\ell)$, then $L(C_p)$ divides $L(C_\ell)$ which gives a contradiction by Lemma \ref{Cp-nmid-Cl}.
\end{proof}

\begin{cor}
	Let $k$ and $\ell $ be positive integers such that $k<\ell $ and $k$ does not divide $\ell$. 
	Then there is no map from $C_\ell$ to $C_k$.
\end{cor}

\begin{proof}
	By the Kleiman-Serre theorem (Theorem \ref{chapman:KleimanSerre}) and Theorem \ref{Cknotdiv}.
\end{proof}

We close this section by remarking again that we do not know if there
is a rational map from $B_{2k}$ to $B_k$ or from $C_{2k}$ to $C_k$.

\section{Remark on the the Curves $y^p-y=x^{p^k+1}+ax$ over $\mathbb F_p$}\label{a}

We  remark that the number of $\mathbb{F}_{p^n}$-rational points on 
$C_{k,a}: y^p-y=x^{p^k+1}+ax$ is equal to the number of $\mathbb{F}_{p^n}$-rational points on 
$C_k$, where $a\in \mathbb F_p^\times$.

All the proofs in this paper go through, with minor changes.
There is the same relationship to $B_k$,  the
map in \eqref{onetoone} must be changed to the map $(x, y)\mapsto (x-a2^{-1},y)$
and all proofs go through.
The curve $X_k$ in the proof of Lemma \ref{divisiblity-odd}
must be replaced by $y^p-y=x^{p^k+1}-a4^{-1}$.

Therefore, the divisibility property (and its opposite direction) also holds for these curves.

%is a one-to-one map over $\mathbb F_{p^n}^2$ and  
%\begin{align*}
%\Tr\left( (x-a2^{-1})^{p^k+1}+(x-a2^{-1})\right)&=
%\Tr\left( x^{p^k+1}-a2^{-1}x^{p^k}+a2^{-1}x-a4^{-1}\right)\\
%&=\Tr(x^{p^k+1})-na4^{-1}.
%\end{align*}

\end{document}